\documentclass[smallextended,envcountsect]{svjour3}

\smartqed
\usepackage{graphicx}
\usepackage{marvosym}
\usepackage{ifsym}
\usepackage{amssymb,epsfig,multirow}
\usepackage{graphicx}
\usepackage{float}
 \usepackage{color}
\usepackage{enumerate}
\usepackage{xcolor}
\usepackage{amsmath}

\journalname{JOTA}
\smartqed  
\usepackage{mathrsfs, lscape, color, float, graphicx, amsfonts, mathrsfs, amssymb,amsmath,amscd, dsfont, bm, algorithm, algorithmic, subfig, multirow}

\begin{document}

\title{Unifying Farkas lemma and S-lemma: new theory and applications in nonquadratic nonconvex optimization}

\author{Meijia Yang$^1$   \and  Yong Xia$^2$   \and  Shu Wang$^3$ }

\authorrunning{M. Yang, Y. Xia, S. Wang} 

\institute{
\Letter Yong Xia, Corresponding author\at yxia@buaa.edu.cn
\and
   Meijia Yang \at myang@ustb.edu.cn           
\and
           Shu Wang \at wangshu.0130@163.com           %
\and
 $1$  School of Mathematics and Physics, University of Science and Technology Beijing, Beijing, 100083, P. R. China  \\
 $2$  State Key Laboratory of Software Development Environment, LMIB of the Ministry of Education, School of Mathematics and System Sciences, Beihang University, Beijing, 100191, P. R. China\\
 $3$ Institute of Computational Mathematics and Scientific/Engineering Computing, State Key Laboratory of Scientific and Engineering Computing, Academy of Mathematics and Systems Science, Chinese Academy of Sciences, Beijing,
100190, P. R. China
   }

\date{Received: date / Accepted: date}

\maketitle

\begin{abstract}
We unify nonlinear Farkas lemma and S-lemma to a generalized alternative theorem for nonlinear nonconvex system. It provides fruitful applications in globally solving nonconvex non-quadratic optimization problems via revealing the hidden convexity.
\end{abstract}

\section{Introduction}
The celebrated Farkas lemma dates back to \cite{F1902}, which characterizes when a linear system has no real solution. Farkas lemma plays a key role in developing strong duality theory of linear programming.
 Following the spirit of applying the separation theorem of two convex sets to prove Farkas lemma,
one can extend Farkas lemma to nonlinear but convex system,  see Theorem 21.1 in \cite{R70}, Section 6.10 in \cite{SW1970}, the survey papers \cite{DJ2014,J2000} and references therein.

The first alternative result for nonconvex system may be due to Finsler \cite{F1937}. There are two homogeneous quadratic forms in a system, with one quadratic form equal to zero and the other one less than or equal to zero. This system has nonzero solution if and only if there is a real constant such that the linear combination of these two matrices is always definite. A strict and more general version of Finsler's theorem, the so-called S-lemma, was first used by Lur'e and Postnikov \cite{LP44} without any theoretical proof at that time. The first proof of S-lemma came out thirty years later due to Yakubovich \cite{Y1971}. Nowadays more proofs of S-lemma are summarized in \cite{P2007}.

S-lemma has many important applications in control theory, robust optimization and convex geometry as well. In nonconvex optimization,  it reveals the hidden convexity of some nonconvex problems including the well-known trust-region subproblem \cite{CGT} and its generalizations, see for example, \cite{J2018,XWS2016}.

There are many generalizations of S-lemma. The most famous one is due to Polyak \cite{P1998}, who established the alternative theorem for the system with three quadratic forms under some assumptions. Jeyakumar et al. \cite{J2009} extended the domain of the system from $\Bbb R^n$ to any linear manifold.  Hu and Huang \cite{HH} generalized S-lemma to even order tenors. Xia et al.  \cite{XWS2016} established necessary and sufficient conditions for S-lemma with equality. As an extension, S-lemma with interval bounds is proposed in \cite{WX2015}. Other generalizations can be found in recent surveys \cite{P2007,DP2006}.

Notice that S-lemma and most of its variations are restricted to nonconvex but quadratic system. The goal of this paper is to establish a general alternative theorem  for nonconvex and non-quadratic system. Following the sprit of  the extension from Farkas lemma to extended Farkas lemma, we intend to combine S-lemma with the nonlinear Farkas lemma. It leads to a unified alternative lemma for nonconvex non-quadratic system, denoted by U-lemma. U-lemma not only generalizes the nonlinear Farkas lemma and S-lemma as special cases, but also plays a more powerful role in helping polynomially solve some nonconvex non-quadratic optimization problems by revealing hidden convexity.

The remainder of this paper is organized as follows. In Section 2, we review some main results which can help better understand our thinking. In Section 3, we establish the unified lemma (U-lemma) by combining  nonlinear Farkas lemma with S-lemma.  In Section 4, we apply the U-lemma to solve two nonconvex nonquadratic optimization problems, the $p$-regularized subproblem and the backward error criterion problem, by revealing its hidden convexity.

\section{Preliminaries}
Throughout the paper we denote by $v(\cdot)$ the optimal value of the problem $(\cdot)$.
Denote by $A\succeq (\preceq)~0$ a positive (negative) semi-definite matrix $A$. Especially, $A\succ 0$ denotes that $A$ is positive definite. Let $I$ be the identity matrix. For a symmetric matrix $A$,  tr$(A)=\sum_{i=1}^nA_{ii}$, $A^+$ denotes the Moore-Penrose generalized inverse of $A$, Range$(A)$ means the range space of $A$ and $\lambda_{\min}(A)$ ($\lambda_{\max}(A)$) denotes the minimal (maximal) eigenvalue of $A$.
$\Bbb R_+=\{x\in\Bbb R:~x\ge0\}$.
 For a real number $a$, sign$(a)$ returns the sign of $a$, i.e., sign$(a)=-1$ if $a<0$, and sign$(a)=1$ otherwise.

The alternative theorem for linear system is the Farkas lemma\cite{F1902}:
\begin{theorem}[Farkas lemma] Let $A\in\Bbb R^{m\times n}$ and $b\in \Bbb R^m$. Then the following two statements are equivalent:
\begin{itemize}
\item[(i)]The linear system $Ax\leq b,~x\geq 0$ has no solution.
\item[(ii)]There exists a nonnegative vector $y\in \Bbb R^m$ such that $A^Ty\geq 0$ and $b^Ty<0$.
\end{itemize}
\end{theorem}

The alternative theorem for convex system is the extended Farkas lemma\cite{R70,SW1970,DJ2014,J2000}.
\begin{theorem}[Extended Farkas lemma]\label{Flem}
Let $f, g_1, \ldots, g_m: \Bbb R^n \rightarrow \Bbb R$ be convex functions and $\Omega$ be a convex set in $\Bbb R^n$. Assume Slater's condition holds for $g_1, \ldots, g_m$,
i.e., there exists an $\widetilde{x}\in relint(\Omega)$ (the set of relative interior points of $\Omega$) such that $g_i(\widetilde{x}) < 0, ~i = 1, \ldots,m$. The following two
statements are equivalent:
\begin{itemize}
\item[(i)] The system $\{f(x)<0,~g_i(x)\le 0, ~i=1,\ldots,m,~x\in \Omega\}$ is not solvable.
\item[(ii)] There exist $\lambda_i\ge 0~(i=1,\ldots,m)$ such that
\[
f(x)+\sum_{i=1}^m \lambda_ig_i(x)\ge 0,~\forall x\in \Omega.
\]
\end{itemize}
\end{theorem}

A general alternative result for nonconvex system is the so-called S-lemma\cite{LP44,Y1971,P2007}
\begin{theorem}[S-lemma]\label{Slem}
Let $f,~g:~\Bbb R^n\rightarrow \Bbb R$ be quadratic functions and suppose that there is an $\widetilde{x}\in \Bbb R^n$ such that $g(\widetilde{x})<0$. Then the following two statements are equivalent:
\begin{itemize}
\item[(i)]There is no $x\in \Bbb R^n$ such that $f(x)<0$, $g(x)\leq 0$.
\item[(ii)]There exists $\mu\geq 0$ such that $f(x)+\mu g(x)\geq 0,~\forall x\in \Bbb R^n$.
    \end{itemize}
\end{theorem}

Yakubovich's proof of S-lemma is fully based on the following hidden convexity of the joint-range for a pair of quadratic forms
\begin{theorem}[Dines, {\rm{\cite{Dines}}}]\label{thm:Dines}
Let $f_H, g_H \in \Bbb R^n\rightarrow \Bbb R$ be quadratic forms, i.e., $f_H=x^TAx,~g_H(x)=x^TBx$ with $A$ and $B$ symmetric matrices, then the set $\{(f_H(x), g_H(x)):~ x\in \Bbb R^n\}\subseteq \Bbb R^2$ is convex.
\end{theorem}
Interestingly, this theorem appears in literature as early as S-lemma itself. Now we are ready to propose our new theorem in the next section.

\section{A unified alternative lemma}
In this section, by combining nonlinear Farkas lemma with S-lemma, we establish a unified alternative lemma (denoted by U-lemma) for more general non-homogeneous nonconvex system.

Throughout this section, we let
\begin{equation}
f(x)=x^TAx+a^Tx+p,~g(x)=x^TBx+b^Tx+q \label{fg}
\end{equation}
be two quadratic functions. Denote by $M= \{(f(x), g(x)):~ x\in \Bbb R^n\}\subseteq \Bbb R^2$ the joint-range for the pair $(f,g)$. Define $M_H(x)=(x^TAx,x^TBx)$ and $M_L(x)=(a^Tx,b^Tx)$.

In order to establish our U-lemma,  we need the following full characterization of the convexity of $M$ rather than Dines's Theorem (Theorem \ref{thm:Dines}), which, as we have mentioned in the introduction,  is crucial in proving S-lemma.
\begin{theorem}[\cite{F2016}]\label{Uconv}
$M$ is convex if, and only if for all $d\in \Bbb R^2,~d=(d_1,d_2)\neq (0,0)$, any of the following
conditions hold:
\begin{itemize}
\item[(A1)] There exists $x$ such that $Ax=0$, $Bx=0$ and $M_L(x)\neq (0,0)$.
\item[(A2)] $d_1B\neq d_2 A$.
\item[(A3)] $-d\not\in M_H(\Bbb R^n)$.
\item[(A4)] There exists $x$ such that $M_H(x)=-d$ and  $d_1 b^Tx= d_2 a^Tx$.
\end{itemize}
\end{theorem}
Theorem \ref{Uconv} contains Dines's Theorem as a special case since if $M_H(\Bbb R^n)\equiv(0,0)$ then $(A3)$ holds and otherwise $(A4)$ automatically holds as $a=b=0$.

\begin{theorem}
[U-lemma]\label{Ulem}
Suppose $f(x)$ and $g(x)$ are two quadratic functions defined in (\ref{fg}). Let  $q_0(z),\ldots,q_m(z)$ be convex functions defined in a convex set $\Omega\subseteq \Bbb R^n$. Assume that there exist $\widetilde{x}\in\Bbb R^n$ and $\widetilde{z}\in {\rm relint} (\Omega)$ such that $g(\widetilde{x})+q_1(\widetilde{z})<0$, $q_i(\widetilde{z})<0,~i=2,\ldots,m$. Then the following two statements are equivalent:
\begin{itemize}
\item[(i)]The system
\[
f(x)+q_0(z)<0,~ g(x)+q_1(z)\leq 0,~ q_i(z)\leq 0,~i=2,\ldots,m,~z\in\Omega
\]
has no solution $(x,z)$.
\item[(ii)]There exist $\mu_i\geq 0,~i=1,2,\ldots,m$ such that
\[
f(x)+\mu_1 g(x)+q_0(z) +\sum_{i=1}^m\mu_i q_i(z)\geq 0,~\forall~x\in \Bbb R^n,z\in \Omega.
\]
\end{itemize}
\end{theorem}
\begin{proof}
We mainly show that (i) implies (ii) as the opposite case is trivial.
We first assume that any of the conditions $(A1)$-$(A4)$ hold.  According to Theorem \ref{Uconv},
the set $M$ is convex and hence
\[
 M_1=\{(f(x),g(x),\underbrace{0,\ldots,0}_{m-1}):~x\in \Bbb R^n\}\subseteq \Bbb R^{m+1}
 \]
  is convex. Define
\begin{eqnarray*}
M_2=\{(v_0,v_1,v_2,\ldots,v_m):~v_0+q_0(z)<0, &&v_i+q_i(z)\leq 0,\\&&i=1,\ldots,m,~z\in \Omega\}
\subseteq \Bbb R^{m+1}.
\end{eqnarray*}
Since $\Omega$ is a convex set and $q_i(z)$ ($i=0,1,\ldots,m$) are all convex functions, it follows that the set $M_2$ is convex.

Under the assumption that Condition $(i)$ holds, the set $M_1$ does not intersect $M_2$. According to the separation theorem for convex sets, $M_1$ and $M_2$ can be separated by a hyperplane, i.e., there exists $\gamma=(\gamma_0,\ldots,\gamma_m)\neq 0$ such that
\begin{equation}
\sum_{i=0}^m \gamma_i v_i \leq \gamma_0 f(x)+\gamma_1 g(x),~
\forall (v_0,v_1,\ldots,v_m)\in M_2,~\forall x\in \Bbb R^n.\label{iq1}
\end{equation}
First we have $\gamma_i\geq 0$ for $i=0,\ldots,m$. Otherwise, there is an index $i$ such that $\gamma_i<0$, letting $v_i\rightarrow-\infty$ yields a contradiction.
Next we show $\gamma_0>0$. Suppose this is not true, i.e.,   $\gamma_0=0$.  Notice that there exist $\widetilde{x}\in\Bbb R^n$ and $\widetilde{z}\in{\rm relint}(\Omega)$ such that $g(\widetilde{x})+q_1(\widetilde{z})<0$ and $q_i(\widetilde{z})<0,~i=2,\ldots,m$. We set $v_0=-\epsilon -q_0(\widetilde{z})$ where $\epsilon>0$ is arbitrarily small. Since $\gamma\neq 0$, $\gamma_0=0$ and  $\gamma_i\geq 0,~i=1,\ldots,m$, we have 
the following two cases to consider:
\begin{itemize}
\item[(a)]If $\gamma_1>0$, taking $x=\widetilde{x}$, $v_1=-q_1(\widetilde{z})$ and $v_i=0~(i=2,\ldots,m)$ in (\ref{iq1}) yields that $\gamma_1 g(\widetilde{x})+\gamma_1 q_1(\widetilde{z})\geq 0$, and hence $g(\widetilde{x})+ q_1(\widetilde{z})\geq 0$, which contradicts the assumption.
\item[(b)] If $\gamma_1=0$, then there is a $k\in\{2,\ldots,m\}$ such that $\gamma_k>0$. Taking $v_k=-q_k(\widetilde{z})$ and $v_i=0~(i\in\{1,\ldots,m\}\setminus\{k\})$ in (\ref{iq1}) yields that $q_k(\widetilde{z})\ge0$, which leads to a contradiction.
\end{itemize}
Therefore, it must hold that $\gamma_0>0$.

By setting $v_0=-q_0(z)-\epsilon$ for any $\epsilon>0$, $v_i=-q_i(z)$ ($i=1,\ldots,m$), $z\in\Omega$ in (\ref{iq1}) and then
dividing both sides of the inequality (\ref{iq1}) by $\gamma_0$, we obtain
\[
f(x)+\mu_1 g(x)+q_0(z)+\sum_{i=1}^m\mu_i q_i(z)+\epsilon\geq 0,~\forall~x\in \Bbb R^n,~z\in \Omega,~\epsilon>0,
\]
where $\mu_i:=\gamma_i/\gamma_0,~i=1,2,\ldots,m$.
Letting $\epsilon\rightarrow 0$ yields
\[
f(x)+\mu_1 g(x)+q_0(z)+\sum_{i=1}^m\mu_i q_i(z)\geq 0,~\forall~x\in \Bbb R^n,~z\in \Omega.
\]

Now, suppose none of the conditions $(A1)$-$(A4)$ holds. One can give a complete description of the nonconvexity of $M$ according to Theorem \ref{Uconv}. As shown in Remark 4.17 in \cite{F2016}, there is a $d=(d_1,d_2)\neq(0,0)$ and a linear transformation from $x$ to $y$ such that
\begin{equation}
\left(f(x),g(x)\right)=\left(\sum_{i=1}^py_i^2-y_{p+1}^2\right)d+
\left(t_1y_1+t_{2}y_{p+1}\right)d_{\perp}+(k_1,k_2),
\label{fgy}
\end{equation}
where $p\in\{0,1,2\ldots,n-1\}$, $d_{\perp}=(e_1,e_2)\neq(0,0)$ such that $e_1d_1+e_2d_2=0$, $t_1\neq 0$, $t_2\neq0$, $k_1$, $k_2$ are real constants, and
\begin{equation}
y_{p+1}^2=1+\sum_{i=1}^py_i^2 ~\Longrightarrow~t_1y_1+t_{2}y_{p+1}\neq 0.\label{yt}
\end{equation}

Suppose $p=0$.
We first assume $d_1\neq 0$. 
Introducing the linear transformation  $\widetilde{y}_1=y_1-(t_1+t_2)e_1/(2d_1)$, we can reformulate (\ref{fgy}) as
\begin{equation}
f(x)=-\widetilde{y}_{1}^2d_1 +\widetilde{k}_1,~
g(x)=-\widetilde{y}_{1}^2d_2 +\widetilde{e}_2\widetilde{y}_{1}+ \widetilde{k}_2,\label{fgy2}
\end{equation}
where $\widetilde{k}_1$, $\widetilde{k}_2$ and $\widetilde{e}_2$ are real constants.

Define two convex functions over $\Bbb R_+$:
\begin{equation}
\widetilde{f}(w):=-d_1w +\widetilde{k}_1, ~\widetilde{g}(w):=-d_2w -|\widetilde{e}_2|\sqrt{w}+ \widetilde{k}_2.\label{fw}
\end{equation}
Suppose Condition $(i)$ holds, we conclude that the system
\begin{equation}
\widetilde{f}(w)+q_0(z)<0,~ \widetilde{g}(w)+q_1(z)\leq 0,~ q_i(z)\leq 0,~i=2,\ldots,m,~w\in\Bbb R_+,z\in\Omega \label{zw}
\end{equation}
has no solution $(w,z)$. Otherwise, for any solution $(w,z)$ satisfying (\ref{zw}), setting $\widetilde{y}_{1}=-{\rm sign}(\widetilde{e}_2)\sqrt{w}$ in (\ref{fgy2}) yields that
  \[
  f(x)=\widetilde{f}(w),~g(x)=-wd_2 -|\widetilde{e}_2|\sqrt{w}+ \widetilde{k}_2
=\widetilde{g}(w),
  \]
which contradicts the assumption that Condition $(i)$ is true.
  According to the extended Farkas lemma (Theorem \ref{Flem}), there exist $\mu_i\geq 0,~i=1,2,\ldots,m$ such that
\[
\widetilde{f}(w)+\mu_1 \widetilde{g}(w)+q_0(z) +\sum_{i=1}^m\mu_i q_i(z)\geq 0,~\forall w\in \Bbb R_+,~\forall z\in \Omega.
\]
It follows from the definitions (\ref{fgy2}) and (\ref{fw}) that for any $x\in\Bbb R^n$, there is a $w=\widetilde{y}_{1}^2$ such that
\begin{equation}
f(x)=\widetilde{f}(w),~g(x)\ge -\widetilde{y}_{1}^2d_2 -|\widetilde{e}_2|\sqrt{\widetilde{y}_{1}^2}+ \widetilde{k}_2
=\widetilde{g}(w). \nonumber
\end{equation}
Therefore, Condition $(ii)$ holds for these nonnegative $\mu_i$ $(i=1,2,\ldots,m)$.

It is similar to study the other case $d_1=0$, which implies that $d_2\neq0$ and  $e_2=0$.  The only difference is that the unique linear term lies in $f$ rather than $g$.

Now we consider $p\ge1$. We show case by case that
\begin{equation}
\inf_{x\in \Bbb R^n}(f(x),g(x))=(-\infty,-\infty). \label{inffg}
\end{equation}
\begin{itemize}
\item $d_1=0$. Then $d_2\neq0$, $e_2=0$ and $e_1\neq 0$.
If $d_2>0$, setting $y_{p+1}=-$sign$(t_2e_1)s$ with $s\rightarrow +\infty$ and $y_i=0$, $\forall i\neq p+1$ yields  (\ref{inffg}). Otherwise, $d_2<0$, setting $y_{1}=-2$sign$(t_1e_1)s$, $y_{p+1}=-$sign$(t_2e_1)s$ with $s\rightarrow +\infty$ and $y_i=0$, $\forall i\neq 1,p+1$, yields  (\ref{inffg}).
\item $d_2=0$. This case is similar to the above case $d_1=0$.
\item $d_1>0,~d_2>0$. Setting $y_{p+1}\rightarrow +\infty$ and $y_i=0$, $\forall i\neq p+1$ yields (\ref{inffg}).
\item $d_1<0,~d_2<0$.
Let $y_{1}\rightarrow +\infty$ and $y_i=0$, $\forall i\neq 1$, we obtain (\ref{inffg}).
\item $d_1d_2<0$.
    Since  $e_1d_1+e_2d_2=0$, we have $e_1e_2>0$. Set  $y_{1}=-$sign$(t_1e_1)s$,  $y_{p+1} =-$sign$(t_2e_2)\sqrt{1+y_1^2}$,
     and $y_i=0$, $\forall i\neq 1,p+1$.
    It follows from (\ref{yt}) that $(t_1y_1+t_{2}y_{p+1})e_i<0$ for any $s>0$. Letting $s\rightarrow +\infty$ leads to (\ref{inffg}).
\end{itemize}
According to (\ref{inffg}) and the existence of $\widetilde{z}$, the system in Condition $(i)$ is always  feasible. Therefore, U-lemma in this case automatically holds.
\end{proof}

\begin{remark}
When $q_i(z)\equiv 0~(i=0,\ldots,m)$, U-lemma reduces to the classical S-lemma (Theorem \ref{Slem}). When $f(x)\equiv0$ and $g(x)\equiv0$, U-lemma reduces to the extended Farkas lemma (Theorem \ref{Flem}).
\end{remark}

\section{Applications}
Our newly established U-lemma fills a gap in alternative theory for non-quadratic nonconvex system. In this section, we show that it has critical applications in globally solving some non-quadratic nonconvex optimization problems by revealing the hidden convexity.
\subsection{Generalized $p$-regularized subproblem with $p>2$}
The generalized $p$-regularized subproblem ($p$-RS) was first proposed in \cite{G2010}:
\begin{eqnarray*}
(p{\rm \text{-}RS)}~~\min_{x\in\Bbb R^n} \left\{h(x)=x^TAx+b^Tx+\rho\|x\|^p\right\},
\end{eqnarray*}
where $\|\cdot\|$ is the standard $l_2$-norm, $p>2$ and $\rho>0$. Without loss of generality, we fix $\rho=1$ in this subsection.
The special case ($3$-RS)  is the well-known Nestrov-Polyak subproblem \cite{N06}, which is an important subproblem  in regularized Newton methods for unconstrained optimization. The double-well potential minimization problem \cite{Fa16} corresponds to  ($4$-RS). If $A\succeq 0$, $h(x)$ is convex as $p>2$. Throughout this subsection, we assume $\lambda_{\min}(A)<0$ so that ($p$-RS) is a nonconvex optimization problem.

Recently,  necessary and sufficient condition for the global minimizer of ($p$-RS) has been established in \cite{HSY2017}. Moreover, in the same paper, it was proved that ($p$-RS) has at most one local non-global minimizer and the  necessary and sufficient condition for the local non-global minimizer is also established.  However, the hidden convexity of ($p$-RS) remains unknown. It was stated in
 \cite{HSY2017}  that ``Notice that, due to the regularization term $\frac{\sigma}{p} \|x\|^p,~p>2$, ($p$-RS) cannot be formulated and solved by a semi-definite program or by polynomial optimization methods.''
The reason behind the  impossibility to reveal the hidden convexity of ($p$-RS)is the lack of mathematical tools to handle non-quadratic nonconvex optimization problems.

Based on U-lemma, now we can show the hidden convexity of ($p$-RS) for $p>2$:
\allowdisplaybreaks[4]
\begin{eqnarray}
&&v(p{\rm\text{-}RS})\nonumber\\
&=&\min_{x,z}\left\{x^TAx+b^Tx+z^{\frac{p}{2}}:~ x^Tx\leq z\right\}\label{prs:0}\\
&=&\sup\left\{t:\left\{(x,z): x^TAx+b^Tx+z^{\frac{p}{2}}-t<0, x^Tx-z\leq 0, z\in\Bbb R_+\right\}=\emptyset\right\}\nonumber\\
&=&\sup\left\{t:\exists\lambda \geq 0: x^TAx+b^Tx+\lambda x^Tx +z^{\frac{p}{2}}-t-\lambda z\geq 0,\forall x\in\Bbb R^n, \forall z\in \Bbb R_+\right\}\nonumber\\\label{prs:1}\\
&=&\sup\left\{t:\exists\lambda \geq 0:
x^TAx+b^Tx+\lambda x^Tx-t+\min_{z\in \Bbb R_+}\{z^{\frac{p}{2}}-\lambda z\}\geq 0,\forall x\in\Bbb R^n\right\}\nonumber\\
&=&\sup\left\{t:\exists\lambda \geq 0:x^TAx+b^Tx+\lambda x^Tx-t+\Phi(\lambda)\geq 0,\forall x\in\Bbb R^n\right\}\nonumber\\
&=&\sup\left\{t:
\left(\begin{array}{cc}A+\lambda I & \frac{1}{2}b\\
\frac{1}{2}b^T & -t+\Phi(\lambda)
\end{array}\right)\succeq 0,~\lambda\geq 0\right\}\label{prs:2}\\
&=&
\left\{\begin{array}{ll}
\sup\limits_{\lambda>-\lambda_{\min}(A)} \Phi(\lambda)-b^T(A+\lambda I)^{-1}b/4,& {\rm if}~b\not\in {\rm Range}(A-\lambda_{\min}(A)I),\\
\sup\limits_{\lambda\ge-\lambda_{\min}(A)} \Phi(\lambda)-b^T(A+\lambda I)^+b/4,& {\rm if}~b \in {\rm Range}(A-\lambda_{\min}(A)I),
\end{array}
\right.
\label{prs:hc}
\end{eqnarray}
where (\ref{prs:1}) is due to U-lemma (Theorem \ref{Ulem}) with the setting $f(x)=x^TAx+b^Tx-t$, $g(x)=x^Tx$, $q_0(z)=z^{\frac{p}{2}}$, $q_1(z)=-z$,  and $\Omega=\Bbb R_+$; (\ref{prs:2}) follows from the fact that
\begin{eqnarray}\label{prs:3}
x^TAx+2a^Tx+c\geq 0,~\forall x\in \Bbb R^n ~\Longleftrightarrow~ \left(\begin{array}{cc}A& a\\
a^T & c
\end{array}\right)\succeq 0;
\end{eqnarray}
(\ref{prs:hc}) holds due to Schur complement theorem,
and
$\Phi(\lambda )$ has an explicit formulation,
\begin{equation}
\Phi(\lambda )=\min_{z\in \Bbb R_+}\{z^{\frac{p}{2}}-\lambda z\}=\lambda^{\frac{p}{p-2}}\left(
\left(2/p\right)^{\frac{p}{p-2}}-
\left(2/p\right)^{\frac{2}{p-2}}\right), \label{Phi:1}
\end{equation}
which is  concave over $\Bbb R_+$ as $p>2$.

We have shown that ($p$-RS) is equivalent to a nonlinear univariate convex optimization problem (\ref{prs:hc}), which can be efficiently and polynomially solved by Newton's method. Let $\lambda^*$ be the optimal solution of (\ref{prs:hc}). Then, it follows from (\ref{Phi:1}) that the optimal $z$-solution of (\ref{prs:0}) is
\[
z^*=\left(2\lambda^*/p\right)^{\frac{2}{p-2}}.
\]
The global optimal solution of ($p$-RS), denoted by $x^*$, can be recovered by solving the trust-region subproblem (\ref{prs:0}) with $z=z^*$, which is easy to solve since the solved $\lambda^*$ serves as an optimal Lagrangian multiplier of (\ref{prs:0}). That is, $x^*=-(A+\lambda^*I)^{-1}b/2$ if $\lambda^*>-\lambda_{\min}(A)$ (the easy case) and otherwise, $x^*$ is any solution of
 the following linear and quadratic equations (the hard case):
\[
2(A+\lambda^*I) x+b=0,~x^Tx=z^*.
\]

\subsection{Backward error criterion}
In numerical analysis, the method with a small backward error is preferred. There are some backward error criteria proposed in \cite{S2003}. The most difficult one is
\begin{eqnarray}
({\rm BE})~~\min_{x\in\Bbb R^{n}} \frac{\|Ax-b\|}{\|A\|\|x\|+\|b\|},\label{cost:1}
\end{eqnarray}
where $A\in\Bbb R^{m\times n}$, $b\in\Bbb R^{m}$, $\|\cdot\|$ is standard $l_2$-norm and spectral norm for vectors and matrices, respectively. We assume $\|A\|>0$. If $Ax=b$ is solvable, then this solution also solves (BE)  and $v$(BE)$=0$. Moreover, if $\|b\|=0$, then $v$(BE)$=\sqrt{\lambda_{\min}(A^TA)}/\|A\|$. In this subsection, we consider the case where $v$(BE)$>0$ and $\|b\|>0$, which means that (BE) is a nontrivial nonconvex fractional programming problem.

With the help of U-lemma, we can show the hidden convexity of (BE) (\ref{cost:1}):
\begin{eqnarray}
&&\min_x \frac{\|Ax-b\|^2}{(\|A\|\|x\|+\|b\|)^2}\nonumber\\
&=&\sup_{t> 0}\left\{t:\{x:\|Ax-b\|^2-t(\|A\|\|x\|+\|b\|)^2<0\}=\emptyset\right\}\nonumber\\
&=&\sup_{t> 0}\left\{t:\{(x,z):\|Ax-b\|^2-t\|A\|^2z-2t\|A\|\|b\|\sqrt{z}-t\|b\|^2<0,
\right.\nonumber\\
&&\left.~~~~~~~~~~~~~~~~~~~~~\|x\|^2\geq z,~z\in\Bbb R_+\}=\emptyset\right\}\nonumber\\
&=&\sup_{t> 0}\left\{t:\exists\lambda\geq 0:\|Ax-b\|^2-t\|A\|^2z-2t\|A\|\|b\|\sqrt{z}-t\|b\|^2 \right. \nonumber\\
&&\left.~~~~~~~~~~~~~~~~~~~~~+\lambda(- \|x\|^2+ z)\geq 0, \forall x\in\Bbb R^{n},~\forall z\in\Bbb R_+\right\}\label{cf:1}\\
&=&\sup_{t> 0}\left\{t:\exists\lambda\geq 0:
\|Ax-b\|^2-t\|b\|^2 -\lambda \|x\|^2\right.\nonumber\\
&&~~~~~~~~~~~~~~~~~~~~~+\min_{z\in\Bbb R_+}\left\{\lambda z-t\|A\|^2z-2t\|A\|\|b\|\sqrt{z}\right\}\geq 0, \forall x\in\Bbb R^{n} \}\label{cf:12}\\
&=&\sup_{t> 0}\left\{t:\|Ax-b\|^2-\lambda\|x\|^2-t\|b\|^2-\frac{
t^2\|A\|^2\|b\|^2}{\lambda-t\|A\|^2}\geq 0, \lambda>t\|A\|^2,~\forall x\in\Bbb R^n\right\}\nonumber\\
&=&\sup_{t> 0}\left\{t:
\left(\begin{array}{cc} A^TA-\lambda I & -A^Tb\\
-b^TA & \|b\|^2-t\|b\|^2-\frac{t^2\|A\|^2\|b\|^2}{\lambda-t\|A\|^2}
\end{array}\right)\succeq 0,~\lambda>t\|A\|^2\right\}\label{cf:2}\\
&=&\sup_{w<\|b\|^2,\lambda>0}\left\{\frac{(w-\|b\|^2)\lambda}{w\|A\|^2-\|A\|^2\|b\|^2-\lambda\|b\|^2}: \left(\begin{array}{cc} A^TA-\lambda I & -A^Tb\\
-b^TA & w
\end{array}\right)\succeq 0\right\} \label{cf:3}\\
&=&1\left/\right.\inf_{w<\|b\|^2,\lambda>0}\left\{
\frac{\|A\|^2}{\lambda}+\frac{\|b\|^2}{\|b\|^2-w}: \left(\begin{array}{cc} A^TA-\lambda I & -A^Tb\\
-b^TA & w
\end{array}\right)\succeq 0 \right\}\nonumber\\
&=&\left\{
\begin{array}{ll}
1\left/\right.\inf\limits_{0<\lambda<\lambda_{\min}(A^TA)}\frac{\|A\|^2}{\lambda}+\frac{\|b\|^2}{\|b\|^2-b^TA
(A^TA-\lambda I)^{-1}A^Tb},& {\rm if}~A^Tb\not\in {\rm R}(A),\\
1\left/\right.\inf\limits_{0<\lambda\le\lambda_{\min}(A^TA)}\frac{\|A\|^2}{\lambda}+\frac{\|b\|^2}{\|b\|^2-b^TA
(A^TA-\lambda I)^{+}A^Tb},& {\rm if}~A^Tb \in {\rm R}(A),
\end{array}
\right. \label{cf:4}
\end{eqnarray}
where (\ref{cf:1}) follows from U-lemma (Theorem \ref{Ulem}) with $f(x)=\|Ax-b\|^2-t\|b\|^2$, $g(x)=-\|x\|^2$, $q_0(z)=-t\|A\|^2z-2t\|A\|\|b\|\sqrt{z}$, $q_1(z)=z$;  (\ref{cf:2}) is due to (\ref{prs:3}); (\ref{cf:3}) holds by introducing
\[
w=\|b\|^2-t\|b\|^2-\frac{t^2\|A\|^2\|b\|^2}{\lambda-t\|A\|^2},
\]
and then $\lambda>t\|A\|^2$ if and only if $w<\|b\|^2$ and $\lambda>0$;  (\ref{cf:4}) is due to Schur complement theorem and R$(A)={\rm Range}(A^TA-\lambda_{\min}(A^TA) I)$.

Problem  (\ref{cf:4}) is a univariate convex optimization problem, which can be solved by Newton's method. Let $\lambda^*$ and $t^*$ be the optimal solution and objective value of  (\ref{cf:4}), respectively. Then the optimal $z$-solution of the minimization subproblem in (\ref{cf:12}) is given by
\[
z^*= \left(t^*\|A\|\|b\|/(\lambda^* -t^*\|A\|^2)\right)^2.
\]
The optimal value of (BE) is:
\[
\min_x \frac{\|Ax-b\|^2}{(\|A\|\|x\|+\|b\|)^2}=t^*.
\]
So (BE) (\ref{cost:1}) is equivalent to the following problem:
\begin{eqnarray*}
&\min_{x\in\Bbb R^n,~z\in \Bbb R_+}& \|Ax-b\|^2-\|A\|^2t^*z-2t^*\|A\|\|b\|\sqrt{z}-t^*\|b\|^2\\
&{\rm s.t.}&\|x\|^2\ge z.
\end{eqnarray*}
Since the objective function is decreasing with respect to $z$, the inequality holds as an equality at the optimal solution. Therefore, $x$ is an optimal solution to (BE) (\ref{cost:1}) if and only if
\[
\|x\|^2=z^*,~\|Ax-b\|^2=t^*(\|A\|\sqrt{z^*}+\|b\|)^2.
\]
To solve  (BE), we first solve two trust-region subproblems
\[
x_m={\rm arg} \min_{\|x\|^2=z^*}\|Ax-b\|^2,~ x_M={\rm arg} \max_{\|x\|^2=z^*}\|Ax-b\|^2.
\]
If either $\|Ax_m-b\|^2$ or $\|Ax_M-b\|^2$ equals to $t^*(\|A\|\sqrt{z^*}+\|b\|)^2$, we are done. Otherwise, define \[
x(\alpha)=\sqrt{z^*}\frac{x_m+\alpha(x_M-x_m)}{\|x_m+\alpha(x_M-x_m)\|}.
\]
Solve the univariate equation
\[
\|Ax(\alpha)-b\|^2=t^*(\|A\|\sqrt{z^*}+\|b\|)^2,~\alpha\in(0,1)
\]
and obtain a solution $\alpha^*$. Then, $x(\alpha^*)$ is an optimal solution of  (BE).

\section{Conclusion}
We unify the nonlinear Farkas lemma and the classical S-lemma to a general alternative theorem, denoted by U-lemma. Some extensions are discussed.
Based on this powerful tool, we reveal hidden convexity of some nonconvex non-quadratic optimization problems,
including the $p$-regularized subproblem ($p>2$), the  backward error criterion problem. It helps to globally solve these nonconvex optimization problems in polynomial time.  We expect for more extensions and fruitful applications, or even beyond the area of optimization.

\begin{acknowledgements}
This research was supported  by National Natural Science Foundation of China
under grants 11822103, 11801173, 11771056 and Beijing Natural Science Foundation Z180005.
\end{acknowledgements}


\end{document}